\documentclass[12pt]{amsart}
\usepackage{amssymb}
\usepackage{mathabx}
\usepackage{xypic}
\usepackage{hhline}
\usepackage{dcpic}
\usepackage{hyperref}

\newtheorem{theorem}{Theorem}

\newtheorem{proposition}[theorem]{Proposition}

\newtheorem{remark}[theorem]{Remark}

\numberwithin{theorem}{section}
\numberwithin{equation}{section}

\begin{document}

\makeatletter
\def\Ddots{\mathinner{\mkern1mu\raise\p@
\vbox{\kern7\p@\hbox{.}}\mkern2mu
\raise4\p@\hbox{.}\mkern2mu\raise7\p@\hbox{.}\mkern1mu}}
\makeatother

\newcommand{\OP}[1]{\operatorname{#1}}
\newcommand{\GO}{\OP{GO}}
\newcommand{\leftexp}[2]{{\vphantom{#2}}^{#1}{#2}}
\newcommand{\leftsub}[2]{{\vphantom{#2}}_{#1}{#2}}
\newcommand{\rightexp}[2]{{{#1}}^{#2}}
\newcommand{\rightsub}[2]{{{#1}}_{#2}}
\newcommand{\AI}{\OP{AI}}
\newcommand{\gen}{\OP{gen}}
\newcommand{\prim}{\OP{\star prim}}
\newcommand{\Image}{\OP{Im}}
\newcommand{\Spec}{\OP{Spec}}
\newcommand{\Ad}{\OP{Ad}}
\newcommand{\tr}{\OP{tr}}
\newcommand{\spec}{\OP{spec}}
\newcommand{\scopy}{\OP{end}}
\newcommand{\ord}{\OP{ord}}
\newcommand{\Cent}{\OP{Cent}}
\newcommand{\wellip}{\OP{w-ell}}
\newcommand{\Nrd}{\OP{Nrd}}
\newcommand{\Res}{\OP{Res}}
\newcommand{\alg}{\OP{alg}}
\newcommand{\LCM}{\OP{LCM}}
\newcommand{\lcm}{\OP{lcm}}
\newcommand{\sgn}{\OP{sgn}}
\newcommand{\SU}{\OP{SU}}
\newcommand{\Hom}{\OP{Hom}}
\newcommand{\Inter}{\OP{Int}}
\newcommand{\diag}{\OP{diag}}
\newcommand{\Sym}{\OP{Sym}}
\newcommand{\GSp}{\OP{GSp}}
\newcommand{\GL}{\OP{GL}}
\newcommand{\GSO}{\OP{GSO}}
\newcommand{\height}{\OP{ht}}
\newcommand{\erf}{\OP{erf}}
\newcommand{\vol}{\OP{vol}}
\newcommand{\cusp}{\OP{cusp,\tau}}
\newcommand{\un}{\OP{un}}
\newcommand{\disci}{\OP{disc,\tau_{\it i}}}
\newcommand{\cuspi}{\OP{cusp,\tau_{\it i}}}
\newcommand{\ellip}{\OP{ell}}
\newcommand{\sph}{\OP{sph}}
\newcommand{\gsimp}{\OP{sim-gen}}
\newcommand{\Aut}{\OP{Aut}}
\newcommand{\disc}{\OP{disc,\tau}}
\newcommand{\sdisc}{\OP{s-disc}}
\newcommand{\aut}{\OP{aut}}
\newcommand{\End}{\OP{End}}
\newcommand{\barQ}{\OP{\overline{\mathbf{Q}}}}
\newcommand{\barQp}{\OP{\overline{\mathbf{Q}}_{\it p}}}
\newcommand{\Gal}{\OP{Gal}}
\newcommand{\GCD}{\OP{GCD}}
\newcommand{\simp}{\OP{sim}}
\newcommand{\pri}{\OP{prim}}
\newcommand{\Normal}{\OP{Norm}}
\newcommand{\Ind}{\OP{Ind}}
\newcommand{\St}{\OP{St}}
\newcommand{\unit}{\OP{unit}}
\newcommand{\reg}{\OP{reg}}
\newcommand{\SL}{\OP{SL}}
\newcommand{\Frob}{\OP{Frob}}
\newcommand{\Id}{\OP{Id}}
\newcommand{\GSpin}{\OP{GSpin}}
\newcommand{\Norm}{\OP{Norm}}

\keywords{Pseudorandom Vectors, Elliptic Curves, Finite Fields, Wiener Process, Monte Carlo Methods}
\subjclass[2010]{11K45, 65C10, 65C05}

\title[Pseudorandom Vector Generation]{Pseudorandom Vector Generation Using Elliptic Curves And Applications to Wiener Processes}
\author{Chung Pang Mok}

\address{School of Mathematical Sciences, Soochow University, 1 Shi-Zhi Street, Suzhou 215006, Jiangsu Province, China}

\email{zpmo@suda.edu.cn}

\maketitle

\begin{abstract}
In this paper we present, using the arithmetic of elliptic curves over finite fields, an algorithm for the efficient generation of a sequence of uniform pseudorandom vectors in high dimensions, that simulates a sample of a sequence of i.i.d. random variables, with values in the hypercube $[0,1]^d$ with uniform distribution. As an application, we obtain, in the discrete time simulation, an efficient algorithm to simulate, uniformly distributed sample path sequence of a sequence of independent standard Wiener processes. This could be employed for use, in the full history recursive multi-level Picard approximation method, for  numerically solving the class of semilinear parabolic partial differential equations of the Kolmogorov type.
\end{abstract}

\section{Introduction}
In numerical integration via the Monte Carlo method, and in the simulation of stochastic processes,  an important role is played by the generation of pseudorandom numbers, and other more general pseudorandom variates. Arguably the most fundamental one is that of  sequence of uniform pseudorandom numbers in the unit interval $[0,1]$, that simulates a sample of a sequence of independent identically distributed random variables, with values in $[0,1]$ with uniform distribution (recall that by the Weyl Criterion for uniform distribution plus the Strong Law of Large Numbers, for a sequence $\{X_n(\cdot)\}_{n \geq 0}$ of independent identically distributed random variables on a probability space $\Omega$, with values in $[0,1]$ with uniform distribution, a sample sequence $\{X_n(\omega)\}_{n \geq 0}$, for $\omega \in \Omega$, is almost surely a uniformly distributed sequence in $[0,1]$). The linear congruential generator is an efficient algorithm to generate such a sequence of uniform pseudorandom numbers. 
\bigskip

In section 2 of the paper, we are concerned with the generation of sequences of uniform pseudorandom vectors in high dimension, that simulate samples of a sequence of independent identically distributed random variables with values in the hypercube $[0,1]^d$ with uniform distribution. These can be generated by using the matrix version of the linear congruential generator; nevertheless, it is well known that, in the higher dimensional case, the sequence of pseudorandom vectors produced by using the linear congruential generator (or its matrix version thereof) could exhibit lattice structures, which sometimes make them not suitable for use in Monte Carlo simulations.
\bigskip

Nonlinear versions of congruential generators could be constructed using the arithmetic of finite fields, for instance the inversive congruential generator (see for example [Ni]); the sequence of pseudorandom vectors thus constructed is observed to be free of lattice structures in general. In this paper we present an algorithm, which relies on the arithmetic of elliptic curves over finite fields, to construct sequence of uniform  pseudorandom vectors. Whereas a finite field is uniquely determined up to isomorphism by its cardinality, one has an ample supply of isomorphism classes of elliptic curves over a given (large) finite field to work with, making it all the more appealing from the perspective of Monte Carlo methods. The algorithm is a variation of that of [Ha], [GBS], having origin in elliptic curve cryptography [Mi], [Kob], [Ka]. Since the applications we have in mind are mainly in Monte Carlo integration and simulation, we modify the original algorithm of {\it loc. cit.} concerning the way pseudorandom vectors are obtained as outputs from points on elliptic curve, in order to achieve high accuracy in these applications. 

\bigskip
In addition, in applications it is important to obtain sequence of pseudorandom vectors with long period. The Law of the Iterated Logarithm plus (multi-dimensional) discrepancy estimates allow us to quantify how long the period has to be, in order that the algorithm outputs good quality sequences of uniform pseudorandom vectors with strong pseudorandomness. The explicit form of the discrepancy estimates are given in section 2.3, in the case where the parameter $e$ of the algorithm is equal to $1$. The proofs of the explicit discrepancy estimates are given in the Appendix (which use the results of [He] and [KS]). We also give the analogue of the Hull-Dobell Theorem for our algorithm to yield the maximum period.

\bigskip

Either the inverse transform method or the Box-Muller method transforms a sequence of independent identically distributed random vectors in $[0,1]^d$ with uniform distribution, to a sequence of independent identically distributed Gaussian random vectors in $\mathbf{R}^d$ with standard normal distribution; in addition, by normalizing these Gaussian random vectors so that they lie on $S^{d-1}$, the $d-1$ dimensional sphere of radius one, what one obtains then is a sequence of independent identically distributed random vectors in $S^{d-1}$ with uniform distribution with respect to $S^{d-1}$ (the uniform measure of $S^{d-1}$ is being normalized so that the measure of $S^{d-1}$ is equal to one, i.e. a probability measure). Now we use the result of [CN], which says that the sequence of uniform probability measures on $S^{d-1}$ converges weakly to the Wiener measure as $d \rightarrow \infty$; more precisely they gave a nonstandard analysis interpretation of the ``Wiener sphere": essentially, by taking $d$ to be a {\it nonstandard infinite integer} $d_{ns}$, the uniform probability measure on $S^{d_{ns}-1}$ corresponds to the Wiener measure, and a random variable taking values in $S^{d_{ns}-1}$ with uniform distribution with respect to $S^{d_{ns}-1}$, corresponds to a standard Wiener process. Thus a sequence of independent identically distributed random vectors in $S^{d_{ns}-1}$ with uniform distribution with respect to $S^{d_{ns}-1}$, corresponds to a sequence of independent standard Wiener processes (the distribution law of standard Wiener process being the Wiener measure). The precise statements will be given in section 3. 
\bigskip

Consequently, when applied to the sequence of uniform pseudorandom vectors in $[0,1]^d$ as constructed in section 2 of the paper, with $d$ being a large integer, this construction gives us a discrete time simulation of, {\it uniformly distributed} sample path sequence of a sequence of independent standard Wiener processes (uniform distribution with respect to discrete time simulation of the Wiener measure). These will be discussed in section 3 of the paper. We will illustrate the algorithm of this paper with Monte Carlo integration in the paper [MZ].
\bigskip

For the class of semilinear second order parabolic partial differential equations of the Kolmogorov type, one has stochastic representation of the viscosity solutions given by Feynman-Kac type formulas [BHJ], namely as a suitable expectation value against the Wiener measure; in a Monte Carlo style, using the full history recursive multi-level Picard approximation method (see for example, [EHJK1], [EHJK2], [HJvW], [HK], [HJKNvW]), these expectation values could be evaluated numerically by employing, the discrete time simulation of uniformly distributed sample path sequence of sequence of independent Wiener processes, as given in this paper. Explicit numerical studies will be the subject of a future investigation.

\section*{Acknowledgement}
The author would like to thank Arnulf Jentzen for his interest in the early version of this work. He would like to thank Professor King Fai Lai and Professor Hourong Qin for encouragements, and also Huimin Zheng for discussions and suggestions related to the contents of the paper. Finally he would like to thank the referee for careful reading and helpful suggestions. 

\section{Construction of sequence of uniform pseudorandom vectors in $[0,1]^d$}

\subsection{R\'esum\'e on elliptic curves over finite fields}
In this subsection we recall some facts concerning elliptic curves over finite fields. For details we refer to chapter V of Silverman's book [Si]. 
\bigskip

Notations: let $p$ be a prime, and $\mathbf{F}_p =\mathbf{Z}/p \mathbf{Z}$ be the finite field with $p$ elements. Let $\overline{\mathbf{F}}_p$ denote the algebraic closure of $\mathbf{F}_p$; for $n \in \mathbf{Z}_{\geq 1}$, denote by $\mathbf{F}_{p^n}$ the unique subfield of $\overline{\mathbf{F}}_p$ consisting of $p^n$ elements ($\mathbf{F}_p$ is the prime subfield of $\mathbf{F}_{p^n}$ with $n=[\mathbf{F}_{p^n}:\mathbf{F}_p] $). Recall that all finite fields of the same cardinality are isomorphic. 
\bigskip

Consider a finite field $F$ of characteristic $p$, with $q$ being the cardinality of $F$ (thus $F \cong \mathbf{F}_q$). An elliptic curve over $F$ could be specified by an affine Weierstrass equation:
\begin{eqnarray}
y^2 + a_1 xy + a_3 y = x^3 + a_2 x^2 + a_4 x +a_6
\end{eqnarray}
with $a_1,a_2,a_3,a_4,a_6 \in F$, whose discriminant $\Delta \in F$ is nonzero (see section III.1 of [Si] for the explicit formulas for the discriminant and also the $j$-invariant associated to an affine Weierstrass equation). The elliptic curve $E$ over $F$ associated to (2.1) is the non-singular projective algebraic curve over $F$ of genus one, defined as the Zariski closure of (2.1) in the projective plane $\mathbf{P}^2$ over $F$. The affine part of $E$ is as given by (2.1), while there is a distinguished point of $E$, the unique point at infinity $\mathbf{O}$ of $E$, that does not belong to the affine part (thus strictly speaking, the elliptic curve is the pair $(E,\mathbf{O})$, but we often refer to it just as $E$ for simplicity). For any field extension $L$ of $F$, we denote by $E(L)$ the set of points of $E$ whose coordinates belong to $L$. We have in particular that $\mathbf{O} \in E(F)$. 

\begin{remark}
\end{remark}
\noindent When $p \geq 5$, any elliptic curve over $F$ is isomorphic over $F$ to one whose affine Weierstrass equation is of the form:
\begin{eqnarray*}
y^2 = x^3 + A x +B 
\end{eqnarray*}
with $A,B \in F$, such that the discriminant $\Delta = -16 \cdot (4 A^3 + 27 B^2) \in F$ is nonzero.

\bigskip

The elliptic curve $E$ is a commutative group variety over $F$ with identity element $\mathbf{O}$. In particular, for any field extension $L$ of $F$, the set $E(L)$ is naturally an abelian group with identity element $\mathbf{O}$. The abelian group addition law on $E$ is given by the chord-tangent law, and the formulas for the addition law on $E$ are given by rational functions of the affine coordinates $x,y$ with coefficients in $F$ (for the explicit formulas see section III.2 of [Si]).
\bigskip

For any field extension $L$ of $F$ and $P,Q \in E(L)$, we denote by $P+Q \in E(L)$ the sum of $P$ and $Q$ with respect to the addition law on $E$, and similarly denote by $-P \in E(L)$ the additive inverse of $P$ with respect to the addition law on $E$. For $k \in \mathbf{Z}_{\geq 1}$ and $P \in E(L)$, we define $[k](  P) \in E(L)$ to be the point given by adding $P$ to itself $k$ times (with respect to the addition law on $E$), and we define $[0] ( P) := \mathbf{O}$, and if $k \in \mathbf{Z}_{<0}$, then $[k] ( P) := [|k|]( - P)$. This is known as the multiplication by $k$ map on $E$.
\bigskip

Since $F$ is a finite field, one has that $E(F)$ is a finite abelian group. By the Hasse bound (cf. Theorem 2.3.1 of Chapter V of [Si]), one has:
\begin{eqnarray}
|\#E(F) - (q+1) | \leq 2q^{1/2}
\end{eqnarray}
while in terms of group structure, one has:
\begin{eqnarray*}
E(F) \cong \mathbf{Z}/M_1 \mathbf{Z} \times \mathbf{Z}/M_2 \mathbf{Z}
\end{eqnarray*}
with $M_1,M_2 \in \mathbf{Z}_{\geq 1}$, $M_1 |M_2$. Thus $E(F)$ is cyclic if and only if $M_1=1$, in which case we say that {\it $E$ is cyclic over $F$}.
\bigskip

As we will see in the next subsection, elliptic curves that are cyclic over $F$ allow us to construct sequences of uniform pseudorandom vectors with maximum period. We recall some of the results of Vladut [Vl]. 
\bigskip

Firstly recall that the elliptic curve $E$ over $F$ is {\it supersingular} ({\it s.s.}) if:
\begin{eqnarray*}
\# E(F) \equiv q+1  \equiv 1 \bmod{p}.
\end{eqnarray*}
If $E$ is supersingular with $j \in F$ being its $j$-invariant, then one has $[\mathbf{F}_p(j):\mathbf{F}_p] =1 \mbox{ or } 2$. See section V.3 of [Si] for other equivalent definitions of supersingularity (in particular, the supersingular property only depends on $E$ over the algebraic closure of $F$); see also section V.4 of {\it loc. cit.} for example of supersingular elliptic curves.
\bigskip

Supersingular elliptic curves exist over any finite field (cf. [Br] for the algorithm for the construction of supersingular elliptic curve over any finite field). Define:
\begin{eqnarray*}
& & c(F)_{ss}  \\
& =& \frac{\mbox{$\#$ of  $F$-isom. classes of s.s. elliptic curves over $F$ cyclic over $F$}}{\mbox{$\#$ of  $F$-isom. classes of s.s. elliptic curves over $F$}}
\end{eqnarray*}
\bigskip

We have the following result of Vladut [Vl]. Recall that the cardinality of $F$ is $q$; the result depends on whether $q$ is a square or a non-square (i.e. on whether or not $\mathbf{F}_{p^2}$ could be embedded into $F$):

\begin{theorem} (Proposition 3.1 of [Vl) 

\bigskip
(i) If $\# F =q $ is not a square, then 
\begin{itemize}
\item $c(F)_{ss} = 1$ for $p=2$ or $p \equiv 1 \bmod{4}$.
\item $c(F)_{ss} = 1/2$ for $p \equiv 3 \bmod{4}$.
\end{itemize}
\,\ (ii) If $\# F =q $ is a square, then 
\begin{itemize}
\item $c(F)_{ss} = 0$ for $p \equiv 1 \bmod{12}$.
\item $c(F)_{ss} = 24/(p+31)$ for $p \equiv 5 \bmod{12}$.
\item $c(F)_{ss} = 24/(p+29)$ for $p \equiv 7 \bmod{12}$.
\item $c(F)_{ss} = 36/(p+49)$ for $p \equiv 11 \bmod{12}$.
\item $c(F)_{ss} = 5/7$ for $p=2$.
\item $c(F)_{ss} = 2/3$ for $p=3$.
\end{itemize}
In particular, unless we both have $\# F$ being a square and $p \equiv 1 \bmod{12}$, there always exists supersingular elliptic curve over $F$ that is cyclic over $F$; in the case where $\# F$ is not a square, we have that all supersingular elliptic curves over $F$ are cyclic over $F$, unless $p \equiv 3 \bmod{4}$.  
\end{theorem}

We also refer to [Vl] for discussions of the general case where one considers elliptic curves that are not necessarily supersingular. 

\subsection{The algorithm, part I}

We follow the formalism of L'Ecuyer [LE]. Given the elliptic curve $E$ over the finite field $F$, the set of states is taken to be $E(F)$. Fix: a nonzero integer $ e$, and $Q \in E(F)$. Define the transition function:
\begin{eqnarray*}
T= T_{e,Q}: E(F) \rightarrow E(F) 
\end{eqnarray*} 
to be the following affine transformation on $E(F)$:
\begin{eqnarray*}
T(P) = [e](  P) +Q, \mbox{  for } P \in E(F).
\end{eqnarray*}

Given an initial state $P_0 \in E(F)$, define the sequence $\{P_n\}_{n \geq 0}$ of points in $E(F)$ recursively by the rule: $P_{n+1} = T(P_n) = [e](P_n)+Q$ for $n \geq 0$. This is the elliptic curve version of the linear congruential generator [Ha], [GBS]. The general formula for $P_n$ for $n \geq 0$ is as follows: firstly if $e=1$, then one has:
\[
P_n = [n](Q) + P_0;
\]
on the other hand, if $e \neq 1$, then one has:
\[
P_n =   [(e^n-1)/(e-1)](Q) + [e^n](P_0).
\]

\bigskip

Put $N := \# E(F)$. Recall that by the Hasse bound, one has $N=q +O(q^{1/2})$. The maximum period for the sequence $\{P_n\}_{n \geq 0}$ is $N$ (in the language of dynamical systems, the maximum period condition amounts to saying that, the dynamical system on the finite state space $E(F)$ defined by the transition function $T_{e,Q}$, is ergodic). We have the following:
\begin{theorem}
The period for the sequence $\{P_n\}_{n \geq 0}$ attains the maximum value $N=\# E(F)$, if and only if the following holds:
\\
(1) $E$ is cyclic over $F$. \\
(2) The point $Q$ has order $N$. \\
(3) For each prime factor $\ell$ of $N$, we have $e \equiv 1 \bmod{\ell}$. \\
(4) If $4 | N$ then $e \equiv 1 \bmod{4}$. 
\end{theorem}
\begin{proof}
For the ``if" part of the proof, assume conditions (1) - (4) hold. Since $E$ is cyclic over $F$, we can fix an isomorphism:
\begin{eqnarray}
E(F) \cong \mathbf{Z} / N \mathbf{Z}
\end{eqnarray} 
and let $\gamma \bmod{N}$ corresponds to the point $Q \in E(F)$ under the isomorphism (2.3). Condition (2) is then equivalent to $\gamma$ being relatively prime to $N$. In addition, under the isomorphism (2.3), the iteration in $E(F)$:
\begin{eqnarray}
P_{n+1} = [e](  P_n) +Q 
\end{eqnarray}
is isomorphic to the following iteration in $\mathbf{Z} / N \mathbf{Z}$:
\begin{eqnarray}
\alpha_{n+1} = e \alpha_n +\gamma \bmod{N}
\end{eqnarray}
which is the iteration appearing in the usual linear congruential generator. By the Hull-Dobell Theorem ([HD] Theorem 1, [Kn] Chapter 3, Theorem A), the conditions (3), (4) on $e$, together with the condition that $\gamma$ is relatively prime to $N$, is equivalent to the condition that, the sequence $\{\alpha_n \bmod{N}\}_{n \geq 0}$ generated by the iteration (2.5) has the maximum period $N$. This finishes the ``if" part of the proof.
\bigskip

For the converse, we first show that $E$ must be cyclic over $F$. Fix an isomorphism:
\begin{eqnarray}
E(F) \cong \mathbf{Z} / M_1 \mathbf{Z} \times \mathbf{Z} / M_2 \mathbf{Z}
\end{eqnarray}
with $M_1 | M_2$ and $N= M_1 M_2$. Let $(\gamma \bmod{M_1},\delta \bmod{M_2})$ correspond to the point $Q \in E(F)$ under the isomorphism (2.6). Then the iteration (2.4) in $E(F)$ corresponds, under the isomorphism (2.6), to the following iteration in $\mathbf{Z} / M_1 \mathbf{Z} \times \mathbf{Z} / M_2 \mathbf{Z}$:
\begin{eqnarray}
& & (\alpha_{n+1} \bmod{M_1},\beta_{n+1} \bmod{M_2})\\
& =& (e \alpha_n+\gamma \bmod{M_1},e \beta_n+\delta \bmod{M_2}). \nonumber
\end{eqnarray}
The period of the sequence:
\[
\{(\alpha_n \bmod{M_1},\beta_n \bmod{M_2})\}_{n \geq 0}
\]
generated by the iteration (2.7), is at most:
\begin {eqnarray*}
\max_{1 \leq a \leq M_1 , 1 \leq b \leq M_2} \lcm(a,b)
\end{eqnarray*}
and this is strictly less than $N$, unless we have $M_1=1, M_2=N$. It follows that $E(F)$ must be cyclic. 

\bigskip
Now with $E(F)$ being cyclic, we argue as in the ``if" part: by using the Hull-Dobell Theorem again, the condition that the period for the sequence $\{P_n\}_{n \geq 0}$ is $N$, implies that conditions (2), (3), (4) must be satisfied. This finishes the ``only if" part of the proof.

\end{proof}

\begin{remark}
\end{remark}
\noindent Efficient algorithms to compute $N=\# E(F)$ were given by Schoof [Sc1], with improvements due to Atkin and Elkies (cf. [El], [Sc2]). 

\bigskip

\begin{remark}
\end{remark}
\noindent We refer to [GI], [Me1] for example, for the study of the generator in the case $e=1$, from the cryptographic perspective, particularly concerning predictability.

\bigskip

\begin{remark}
\end{remark}
\noindent One could consider more general kind of higher order linear recursive sequence of points on elliptic curves over finite fields, as in [GL] (i.e. the elliptic curve version of linear-feedback shift register sequence). In other words, fix $\mathfrak{L} \in \mathbf{Z}_{\geq 0}$, integers $e_0,\cdots, e_{\mathfrak{L}}$ not all zero, and $Q \in E(F)$. Then given seeds $P_0,\cdots,P_{\mathfrak{L}} \in E(F)$, one defines the sequence $\{P_n\}_{n \geq 0}$ recursively by:
\[
P_{n} = [e_0](P_{n-1}) + \cdots + [e_{\mathfrak{L}}](P_{n-1-\mathfrak{L}}) + Q, \,\  n \geq \mathfrak{L}+1.
\]
In addition one could also replace the integer multiplication maps by other types of endomorphisms of elliptic curves over finite fields (for instance the Frobenius endomorphism) in defining the recursive sequence, as in [Me2]; cf. [Koh] for the computation of endomorphism rings of elliptic curves over finite fields.

\subsection{The algorithm, part II}

As before $F$ is a finite field, with cardinality $q=p^m$. Let $m = a\cdot r$ be a factorization of $m$ , with $a,r \in \mathbf{Z}_{\geq 1}$. Again we follow the formalism of [LE]. The set of outputs is going to be a subset of $[0,1]^{2r}$, and we define the output function $G: E(F) \rightarrow [0,1]^{2r}$ in this subsection.
\bigskip

Since $a| m$, there is a unique finite subfield $K$ of $F$ with cardinality of $K$ being equal to $p^a$; thus $[K:\mathbf{F}_p]=a$ and $[F:K]=r$. Fix a basis $\mathfrak{a} =\{\kappa_1,\cdots, \kappa_a\}$ of the extension $K/\mathbf{F}_p$. For each element $\zeta \in K$ and $i=1,\cdots, a$, define $\phi_{i}(\zeta)  \in \{0,1,\cdots, p-1 \} \subset \mathbf{Z}_{\geq 0}$, to be the coordinates of $\zeta$ with respect to the basis $\mathfrak{a}$ (thus $\zeta = \phi_1(\zeta) \cdot \kappa_1 + \cdots + \phi_a(\zeta) \cdot \kappa_a$); here we fix $\{0,1,\cdots,p-1 \}$ to be the set of representatives of elements of $\mathbf{F}_p$. 
\bigskip

Define the map:
\begin{eqnarray*}
\Phi : K \rightarrow [0,1) \subset [0,1]
\end{eqnarray*}
\begin{eqnarray*}
\Phi(\zeta) = \sum_{i=1}^a \frac{\phi_i(\zeta)}{p^i}, \,\ \zeta \in K.
\end{eqnarray*}
Note that $\Phi$ is injective. 
\bigskip

Fix also a basis $\mathfrak{b}=\{\lambda_1,\cdots,\lambda_r \}$ of the extension $F/K$. For $\eta \in F$ and $j=1,\cdots,r$, define $\langle \eta \rangle_j  \in K $ to be the coordinates of $\eta$ with respect to the basis $\mathfrak{b}$ (thus we have $\eta =  \langle \eta \rangle_1 \cdot \lambda_1 + \cdots + \langle \eta \rangle_r \cdot \lambda_r$). 

\bigskip
We now define the output function $G$ (which depends on the choices of the bases $\mathfrak{a}$ and $\mathfrak{b}$):  for $P \in E(F)$ with $P \neq \mathbf{O}$,  let $x(P),y(P) \in F$ be the affine Weierstrass coordinates of the point $P$; define $G(P) \in [0,1]^{2r}$ to be the vector:
\begin{eqnarray*}
& & G(P) \\
&  = &\Big(\Phi(  \langle x(P) \rangle_1)  , \cdots, \Phi(  \langle x(P) \rangle_r) ,\Phi(  \langle y(P) \rangle_1),  \cdots ,\Phi(  \langle y(P) \rangle_r) \Big )
\end{eqnarray*}
which actually lies in $[0,1)^{2r}$. Finally we define $G(\mathbf{O}) = (1,\cdots,1) \in [0,1]^{2r}$ to be the vector where all the coordinates are $1$. It is clear that the output function $G: E(F) \rightarrow [0,1]^{2r}$ thus defined is injective. 
\bigskip

Given a value of initial state $P_0 \in E(F)$, we then compute the sequence of vectors $G(P_n) \in [0,1]^{2r}$ for $n \geq 0$ (with the sequence of points $P_n \in E(F)$ for $n \geq 0$, being defined as in the previous subsection). The sequence $\{G(P_n)\}_{n \geq 0}$ is our construction of sequence of uniform pseudorandom vectors in $[0,1]^{2r}$.

\bigskip

In computations, it is useful to note that the quantities appearing in the definition of the vector $G(P)$, for $P \in E(F)$, could be rewritten as follows. With notations as above, denote by $\mathfrak{b}^{\prime} =\{\lambda_1^{\prime},\cdots,\lambda_r^{\prime}\}$ the basis of $F/K$ that is dual to $\mathfrak{b} = \{ \lambda_1,\cdots,\lambda_r \}$ with respect to $Tr_{F/K}$, the trace from $F$ to $K$; thus for $1 \leq j,j^{\prime}\leq r$, one has $Tr_{F/K}(\lambda_j \cdot \lambda_{j^{\prime}}^{\prime})=1$ (as elements of $K$) if $j=j^{\prime}$, and is equal to $0$ if $j \neq j^{\prime}$. Then for $\eta \in F$ and $j=1,\cdots,r$, one has
\[
\langle \eta \rangle_j  = Tr_{F/K}(\eta \cdot \lambda_j^{\prime}).
\]

Similarly denote by $\mathfrak{a}^{\prime} =\{\kappa_1^{\prime},\cdots,\kappa_a^{\prime}  \}$ the basis of $K/\mathbf{F}_p$ that is dual to $\mathfrak{a}=\{\kappa_1,\cdots,\kappa_a\}$ with respect to $Tr_{K/\mathbf{F}_p}$, the trace from $K$ to $\mathbf{F}_p$. Then for $\zeta \in K$ and $i=1,\cdots,a$, one has
\[
\phi_i(\zeta) = Tr_{K/\mathbf{F}_p}(\zeta \cdot \kappa_i^{\prime})
\]
where we understood that the value of the trace to $\mathbf{F}_p$ is taken as an element in  $\{0,1,\cdots ,p-1\} \subset \mathbf{Z}_{\geq 0}$.
\bigskip

Now for any $\eta \in F$, and $ 1 \leq j \leq r$,  one has:
\begin{eqnarray}
& & \Phi(  \langle  \eta \rangle_j )   \\
&=& \sum_{i=1}^a \frac{Tr_{K/\mathbf{F}_p} (   Tr_{F/K}(\eta \cdot \lambda^{\prime}_j)  \cdot \kappa^{\prime}_i   )  }{p^i} \nonumber\\
&=& \sum_{i=1}^a \frac{   Tr_{K/\mathbf{F}_p} (   Tr_{F/K} (  \eta \cdot \lambda^{\prime}_j \cdot \kappa^{\prime}_i ))      }{p^i}  \nonumber \\
&=& \sum_{i=1}^a \frac{Tr_{F/\mathbf{F}_p} (\eta \cdot \lambda^{\prime}_j \cdot \kappa^{\prime}_i) }{p^i}. \nonumber
\end{eqnarray}
And so the quantities appearing in the definition of the vector $G(P)$, could be computed directly by using the trace from $F$ to $\mathbf{F}_p$.

\bigskip

In applications it is also important to obtain sequence of pseudorandom vectors with long period. For instance, under the conditions of Theorem 2.3, we have, for any value of initial state $P_0 \in E(F)$, that the sequence $\{G(P_n)\}_{n \geq 0}$ has maximum period (equal to $N= \# E(F)$).
\bigskip

Finally, one way to justify the claim that the sequence $\{ G(P_n) \}_{n \geq 0}$ simulates a sample sequence, of a sequence of random variables with values in $[0,1]^{2r}$ with uniform distribution, is to estimate the discrepancy of the sequence $\{ G(P_n) \}_{n \geq 0}$. 

\bigskip
In general, cf. Chapter 2 of [Ni], the (extreme) discrepancy of a nonempty finite set $\mathcal{M} \subset [0,1)^{h}$ (the definition depends on the dimension $h \in \mathbf{Z}_{\geq 1}$) is defined as:
\begin{eqnarray*}
\mathcal{D}=\mathcal{D}(\mathcal{M})= \sup_{\mathcal{B} \subset [0,1)^{h}} \Big| \frac{\#(\mathcal{B} \cap \mathcal{M})}{\# \mathcal{M}} - \mbox{volume}( \mathcal{B}) \Big|
\end{eqnarray*}
where the $\sup$ is taken over all the rectangular boxes $\mathcal{B} \subset [0,1)^{h}$ of the form $\prod_{\iota=1}^h[\mu_{\iota},\nu_{\iota})$.

\bigskip
Now for simplicity we consider the case where the integer $e$ in the definition of the transition function $T=T_{e,Q}$, is equal to $1$. Let $t$ be the order of the point $Q$, which is thus equal to the period of the sequence $\{P_n\}_{n \geq 0}$ (hence also that of $\{G(P_n)\}_{n \geq 0}$); we define $t^{\prime}$ to be equal to $t-1$ if one has $P_n=\mathbf{O}$ (i.e. $G(P_n) =G(\mathbf{O}) =(1,\cdots,1)$) for some $0 \leq n \leq t-1$, and is equal to $t$ otherwise. Define $\mathcal{M} \subset [0,1)^{2r}$ to be the set of points $G(P_n)$ for $0 \leq n \leq t-1$, with the point $G(P_n)$ being discarded if it is equal to $(1,\cdots,1)$. The cardinality of $\mathcal{M}$ is thus equal to $t^{\prime}$. The discrepancy $\mathcal{D}$ of the sequence $\{G(P_n)\}_{n \geq 0}$ is then defined to be the discrepancy of the set $\mathcal{M} \subset [0,1)^{2r}$. 

\bigskip
Using Theorem 1 and Corollary 4 of [He], concerning the general discrepancy estimates with respect to the base $p$ Walsh function system (a variant of the Erd\"os-Tur\'an-Koksma inequality), together with the exponential sum estimates of [KS], the argument in the proof of Theorem 1 of [ES1] can be generalized to give the following bound:
\begin{eqnarray}
\mathcal{D}  \leq 1 -  \big(1- \frac{1}{p^a}\big)^{2r}  +  \frac{4q^{1/2}}{t^{\prime}} \big( 2.43\ln (p^a)  +1\big)^{2r} .
\end{eqnarray}

\bigskip
More generally, for $2 \leq s \leq  t$, define $t^{\prime \prime}$ to be equal to $t-s$ if one has $P_n=\mathbf{O}$ for some $0 \leq n \leq t-1$, and is equal to $t$ otherwise.  Define the $s$-discrepancy $\mathcal{D}_s$ of the sequence $\{G(P_n)\}_{n \geq 0}$, which is a measure of the statistical independence of $s$ successive terms in $\{G(P_n)\}_{n \geq 0}$, as the discrepancy of the following set $\mathcal{N} \subset [0,1)^{2rs}$, consisting of points:
\[
 \big( G(P_{n}),  G(P_{n+1}),\cdots,G(P_{n+s-1}) \big), \,\ 0 \leq n \leq t-1
\]
regarded as vectors in $[0,1)^{2rs}$; here the vector is discarded if one of the components $G(P_n),\cdots,G(P_{n+s-1})$ is equal to $(1,\cdots,1)$. The cardinality of $\mathcal{N}$ is equal to $t^{\prime \prime}$. 

\bigskip

One also has the following non-overlapping variant of $\mathcal{D}_s$: define $\widetilde{\mathcal{D}}_s$ for $2 \leq s \leq t$, as the discrepancy of the following set $\widetilde{\mathcal{N}} \subset [0,1)^{2rs}$, consisting of points:
\[
 \big( G(P_{ns}),  G(P_{ns+1}),\cdots,G(P_{ns+s-1}) \big), \,\ 0 \leq n \leq \frac{t}{\gcd(s,t) }-1
\]
again regarded as vectors in $[0,1)^{2rs}$; here as before the vector is discarded if one of the components $G(P_{ns}),\cdots,G(P_{ns+s-1})$ is equal to $(1,\cdots,1)$. The cardinality of $\widetilde{\mathcal{N}}$ is equal to $t^{\prime \prime}/\gcd(s,t) $.

\bigskip
Then similarly the argument in the proof of Theorem 3 of [HS], can be generalized to give the bounds for $2 \leq s \leq t$:

\bigskip
\noindent Assume $p \geq 5$, then:
\begin{eqnarray}
 \mathcal{D}_s   \leq    1 -  \big(1- \frac{1}{p^a}\big)^{2rs}  +  \frac{6 q^{1/2}s }{t^{\prime \prime}} \big( 2.43\ln (p^a)  +1\big)^{2rs} .
  \end{eqnarray}

\noindent Assume $p \geq 5$ and $\gcd(s,p)=1$, then:
 \begin{eqnarray}
\widetilde{\mathcal{D}}_s   \leq  1 -  \big(1- \frac{1}{p^a}\big)^{2rs}  +  \frac{6 q^{1/2}s^3  }{t^{\prime \prime}} \big( 2.43\ln (p^a)  +1\big)^{2rs}.
\end{eqnarray}
For completeness, we give the details for the proof of (2.9), (2.10) and (2.11) in the Appendix.

\bigskip
The term $ 1 -  \big(1- 1/p^a\big)^{2r} $ occurring in (2.9), and respectively the term $1 -  \big(1- 1/p^a\big)^{2rs} $ occurring in (2.10) and (2.11), is a discretization error term (cf. [He]), and we shall ignore it for the purpose of the discussion of the present moment. Thus, with $r$ and $s$ being fixed (and noting that $p^a=q^{1/r}$), if we have $t \gg q^{1/2 + \epsilon}$ (with $\epsilon >0$), then the sequence $\{ G(P_n)\}_{n \geq 0}$ can be regarded as a good quality sequence of uniform pseudorandom vectors in $[0,1]^{2r}$; in addition, in view of the Law of the Iterated Logarithm ({\it c.f.} Chapter 7 of [Ni]), the sequence $\{ G(P_n)\}_{n \geq 0}$ exhibits strong pseudorandomness if one has $t=q^{1+o(1)}$ (which certainly holds when $t=N$, by the Hasse inequality).

\bigskip

 For other values of $e$, if the period $t$ of the sequence $\{G(P_n)\}_{n \geq 0}$ is equal to $N$, then the estimate (2.9) for the discrepancy $\mathcal{D}$ of the sequence $\{ G(P_n)\}_{n \geq 0}$ again holds with $t=N$ (and so $t^{\prime}=N-1$); this is because when the period is equal to $N$, then by Theorem 2.3, we have in particular that $E(F)$ is cyclic, and so the estimate for the discrepancy $\mathcal{D}$ is reduced to the case where $e=1$ and $Q$ is a point of order $N$. An interesting problem is to obtain estimates for the discrepancy $\mathcal{D}$ in the situation when $k$ is not equal to $1$ and when the period is not equal to $N$ (and similarly the problem of obtaining estimates for $\mathcal{D}_s$ and $\widetilde{\mathcal{D}}_s$ for $2 \leq s \leq t$, in the situation when $e$ is not equal to $1$); in the particular case where $Q=\mathbf{O}$ (also known as the elliptic curve version of the power generator), we refer to the papers [LS, BFGS, ES2, AS, Me]. 

\bigskip

Summarizing, our general algorithm for the construction of sequence of uniform pseudorandom vectors $\{u_n\}_{n \geq 0}$ in $[0,1]^d$ for $d \in \mathbf{Z}_{\geq 1}$ is as follows:

\bigskip

\noindent {\bf Algorithm}

\begin{itemize}

\item Fix a prime $p$. 

\item Fix $a,r,s \in \mathbf{Z}_{\geq 1}$ with $d \leq 2rs$. Put $m=a \cdot r$. 

\item Fix finite field $F$ with $\# F =q= p^m$, and let $K$ be the unique subfield $F$ with $ \# K =p^a$. Fix a basis $\mathfrak{a}$ of the extension $K/\mathbf{F}_p$, and a basis $\mathfrak{b}$ of the extension $F/K$.

\item Fix an elliptic curve $E$ over $F$. 

\item Fix nonzero integer $e$ and $Q \in E(F)$ to define the transition function $T=T_{e,Q}: E(F) \rightarrow E(F)$. The output function $G: E(F) \rightarrow [0,1]^{2r}$ is defined as above using the bases $\mathfrak{a}$ and $\mathfrak{b}$.  

\item Given an initial state $P_0 \in E(F)$, compute, for $n \geq 0$, the sequence of points $P_n \in E(F)$ by iterating the transition function $T$.

\item For $n \geq 0$, compute the vector $G(P_n) \in [0,1]^{2r}$. 

\item For $n \geq 0$, form the vector $ \mathfrak{u}_n = \big( G(P_{ns}), G(P_{ns+1}) ,\cdots,G(P_{ns +s-1})  \big)$, with $\mathfrak{u}_n$ being regarded as a vector in $[0,1]^{2rs}$.

\item Fix a set injection $\pi:  \{1,\cdots,d \}\rightarrow \{ 1,\cdots,2rs \}$.

\item For $n \geq 0$, define $u_n = (u_n^{(1)},\cdots,u_n^{(d)})$ to be the vector in $[0,1]^d$, by taking $u_n^{(i)}$, for $1 \leq i \leq d$, to be the $\pi(i)$-th coordinate of  $\mathfrak{u}_n$.
\end{itemize}

\section{Discrete time simulation of uniformly distributed sample path sequence of a sequence of independent Wiener processes}

\subsection{The algorithm, part I}
To construct discrete time simulation of uniformly distributed sample path sequence of a sequence of independent standard Wiener processes, we first transform a sequence $\{u_n\}_{\geq 0}$ of uniform pseudorandom vectors in the unit hypercube $[0,1]^d$, as constructed by the algorithm of the previous section, to a sequence of Gaussian pseudorandom vectors in $\mathbf{R}^d$ with standard normal distribution (i.e. mean vector $=$ the zero vector in $\mathbf{R}^d$, variance matrix $=$ the $d \times d$ identity matrix). Firstly, we delete any vectors from the sequence $\{ u_n \}_{\geq 0}$ whose any coordinate is either $0$ or $1$. When this is done, we may then assume without loss of generality that $\{u_n\}_{\geq 0}$ is a sequence of uniform pseudorandom vectors in $(0,1)^d$, with $u_n = (u_n^{(1)},\cdots,u_n^{(d)})$ (here $u_n^{(i)} \in (0,1)$ for $i=1,\cdots d$).

\bigskip
\noindent 1. The Inverse Transform Method: let 
\begin{eqnarray*}
\Psi(x) = \frac{1}{\sqrt{2 \pi}} \int^x_{-\infty} e^{- t^2/2} dt  = \frac{1}{2} \Big(  1+ \erf \big(\frac{x}{\sqrt{2}}\big)\Big)
\end{eqnarray*}
be the standard normal cumulative distribution function (here $\erf$ is the error function). Put for $n \geq 0$ and $i=1,\cdots,d$:
\[
v_n^{(i)} = \Psi^{-1} \big( u_n^{(i)} \big).
\]
The sequence $\{v_n\}_{n \geq 0}$, with $v_n=(v_n^{(1)},\cdots, v_n^{(d)}) $, is then a sequence of Gaussian pseudorandom vectors in $\mathbf{R}^d$, with standard normal distribution on $\mathbf{R}^d$. 

\bigskip
\noindent Justification: Let $\{U_n (\cdot)\}_{\geq 0} $ be a sequence of independent identically distributed random variables with values in $(0,1)^d$ with uniform distribution, with $U_n(\cdot)=(U_n^{(1)}(\cdot),\cdots,U_n^{(d)}(\cdot))$. Put for $n \geq 0$ and $i=1,\cdots d$:  
\begin{eqnarray*}
 V_n^{(i)}(\cdot) = \Psi^{-1}(U_n^{(i)}(\cdot)). 
\end{eqnarray*}
Then the sequence $\{V_n (\cdot)\}_{\geq 0}$, with $V_n(\cdot)=(V_n^{(1)}(\cdot),\cdots,V_n^{(d)}(\cdot))$, is a sequence of independent identically distributed random variables with values in $\mathbf{R}^d$, with standard normal distribution on $\mathbf{R}^d$.

\bigskip

\noindent 2. The Box-Muller Method: (without loss of generality) assume that $d$ is even: $d= 2 g$. Put for $n \geq 0$ and $j=1,\cdots, g$:
\begin{eqnarray*}
v_n^{(2j-1)} & =& \sqrt{- 2 \ln( u_n^{(2j-1)})}  \cos(2 \pi u_n^{(2j)})  \\  
v_n^{(2j)} &=& \sqrt{- 2 \ln( u_n^{(2j-1)})}  \sin(2 \pi u_n^{(2j)}).  
\end{eqnarray*}
The sequence $\{v_n \}_{\geq 0}$, with $v_n = (v_n^{(1)},\cdots,v_n^{(d)}) $, is then a sequence of Gaussian pseudorandom vectors in $\mathbf{R}^d$, with standard normal distribution on $\mathbf{R}^d$. 

\bigskip
\noindent Justification: Let $\{U_n(\cdot) \}_{\geq 0} $ be a sequence of independent identically distributed random variables with values in $(0,1)^d$ with uniform distribution, with $U_n(\cdot)=(U_n^{(1)}(\cdot),\cdots,U_n^{(d)}(\cdot))$. Put for $n \geq 0$ and $j=1,\cdots,g$: \begin{eqnarray*}
V_n^{(2j-1)}(\cdot) & =& \sqrt{- 2 \ln( U_n^{(2j-1)}(\cdot))}  \cos(2 \pi U_n^{(2j)}(\cdot))  \\  
V_n^{(2j)} (\cdot)&=& \sqrt{- 2 \ln( U_n^{(2j-1)}(\cdot))}  \sin(2 \pi U_n^{(2j)}(\cdot))
\end{eqnarray*}
then the sequence $\{V_n(\cdot)\}_{\geq 0}$, with $V_n(\cdot)  = (V_n^{(1)}(\cdot),\cdots,V_n^{(d)}(\cdot))$, is a sequence of independent identically distributed random variables with values in $\mathbf{R}^d$, with standard normal distribution on $\mathbf{R}^d$ ([BM], [OG]).

\bigskip

Now let $S^{d-1} \subset \mathbf{R}^d$ be the $d-1$ dimensional unit sphere consisting of elements whose norm is equal to one. The uniform measure on $S^{d-1}$ (i.e. rotationally invariant measure) is normalized to be equal to one, i.e. a probability measure. 

\bigskip

Then given a sequence of Gaussian pseudorandom vectors in $\mathbf{R}^d$ with standard normal distribution, we delete any vector from the sequence that is equal to the zero vector. When this is done, we may then assume that none of the $v_n$ is equal to the zero vector. Put for $n \geq 0$: $w_n = v_n / || v_n|| \in S^{d-1}$ (where $ || \cdot ||$ denotes the norm of a vector). The sequence $\{ w_n \}_{n \geq 0}$ is then a sequence of pseudorandom vectors in $S^{d-1}$, with uniform distribution with respect to $S^{d-1}$.

\bigskip
\noindent Justification: Let $\{V_n (\cdot)\}_{\geq 0} $ be a sequence of independent identically distributed random variables with values in $\mathbf{R}^{d}$ with standard normal distribution. Then the sequence $\{W_n(\cdot) \}_{n \geq 0}$, with $W_n(\cdot) :=V_n(\cdot) / || V_n(\cdot)||$ (assuming that $V_n(\cdot)$ does not take the zero vector as value for any $n \geq 0$), is a sequence of independent identically distributed random variables with values in $S^{d-1}$, with uniform distribution respect to $S^{d-1}$ (this follows from the fact that the standard normal distribution on $\mathbf{R}^d$ is invariant with respect to rotation about the origin). 

\subsection{The algorithm, part II}

We can now construct the discrete time simulation of uniformly distributed sample path sequence, of a sequence of independent standard Wiener processes (to be precise, uniform distribution with respect to discrete time simulation of the Wiener measure). Our construction relies on work of Cutland-Ng [CN] (which is based in turn on [Cu]). We consider (one dimensional) standard Wiener processes for the time interval $[0,T]$ with $T$ a positive real number. Without loss of generality we take $T=1$; indeed recall the scaling invariance property of Wiener process: if $\mathcal{X}: \Omega \times [0,1] \rightarrow \mathbf{R}$ is a standard Wiener process on the time interval $[0,1]$, then the stochastic process $\mathcal{X}^T: \Omega \times [0,T] \rightarrow \mathbf{R}$ as given by:
\[
\mathcal{X}^T(\omega,t) = T^{1/2 }\cdot \mathcal{X}(\omega,t/T), \,\ \omega \in \Omega, \,\ t \in [0,T]
\]
is a standard Wiener process on the time interval $[0,T]$. Denote by $C_0([0,1])$ the set of $\mathbf{R}$-valued continuous functions $c(t)$ on $[0,1]$ such that $c(0)=0$. The set $C_0([0,1])$ is equipped with the Wiener measure (and hence is a probability space).

\bigskip

With $d \geq 1$ as before, put $t_i = i/d$ for $i=0,1,\cdots , d$ (discretization of the time interval $[0,1]$). Define the map:
\[
\Sigma_d: S^{d-1} \rightarrow C_0([0,1])
\]  
as follows: for $w  = (w^{(1)},\cdots,w^{(d)})\in S^{d-1}$, define $\Sigma_d (w): [0,1] \rightarrow \mathbf{R}$ to be the polygonal path in $C_0([0,1])$, such that:
\begin{eqnarray*}
(\Sigma_d (w))(0)  &=& 0 \\
(\Sigma_d (w))(t_i)  &=& \sum_{k=1}^{i} w^{(k)} \mbox{ for } i=1,\cdots, d
\end{eqnarray*}
and $(\Sigma_d (w))(t)$ is linearly interpolated between $(i-1)/d \leq  t \leq i/d$ for $i=1,\cdots,d $. The map $\Sigma_d$ is clearly injective and measurable. 

\bigskip

Firstly, we have: 
\begin{theorem} (Theorem 2.4 of [CN])
For $d \geq 1$ let $m_d$ be the uniform probability measure on $S^{d-1}$. Then the sequence of measures on $C_0([0,1])$ given by the push-forward of $m_d$ to  $C_0([0,1])$:
\[
m_d \circ \Sigma_d^{-1}
\]
converges weakly to the Wiener measure on $C_0([0,1])$, as $d \rightarrow \infty$.
\end{theorem}

\bigskip

The measure $m_d \circ \Sigma_d^{-1}$ could thus be regarded as a discrete time simulation of the Wiener measure on $C_0([0,1])$.

\bigskip

Now if $\{W_n(\cdot)\}_{n \geq 0}$ is a sequence of independent identically distributed random variables on a probability space $\Omega$ with values in $S^{d-1}$, with uniform distribution with respect to $S^{d-1}$, then by the Weyl Criterion for uniform distribution plus the Strong Law of Large Numbers, a sample sequence $\{W_n(\omega)\}_{n \geq 0}$ of $\{W_n (\cdot)\}_{n \geq 0}$ (for $\omega \in \Omega$), is almost surely, a sequence in $S^{d-1}$ with uniform distribution with respect to $S^{d-1}$ (cf. [KN], Chapter 3, Theorem 2.2, in the general setting of compact Hausdorff topological space with countable base, with respect to Borel probability measure); as in the previous subsection, this could be simulated by a sequence $\{w_n\}_{n \geq 0}$ of pseudorandom vectors in $S^{d-1}$ with uniform distribution with respect to $S^{d-1}$. By taking $d$ to be a large integer, the sequence $\{ \Sigma_d (w_n) \}_{n \geq 0}$ could be considered as discrete time simulation of uniformly distributed sample path sequence, of a sequence of independent standard Wiener processes. Here since we are dealing with discrete time simulation, uniform distribution here is in fact meant to be with respect to discrete time simulation of the Wiener measure, namely $m_d \circ \Sigma_d^{-1}$. 

\bigskip
To justify this we use the more precise form of the result of Cutland-Ng [CN], which used the tools of nonstandard analysis. We refer to [Go] for the background on nonstandard analysis.

\bigskip

Let $\leftexp{*}{\mathbf{R}}$ be nonstandard extension of $\mathbf{R}$, and for $x \in \leftexp{*}{\mathbf{R}}$ define $\leftexp{\circ}{x}$ to be the standard part of $x$. This means that, if $x$ is finite, then $\leftexp{\circ}{x}$ is the unique element in $\mathbf{R}$ that is infinitesimally close to $x$; otherwise if $x$ is not finite, then we simply define $\leftexp{\circ}{x}$ to be $\pm \infty$. Fix a nonstandard infinite integer $d_{ns}$. Inside the $\leftexp{*}{\mbox{Euclidean}}$ space $\leftexp{*}{\mathbf{R}}^{d_{ns}}$, denote by $\mathbf{S}  \subset \leftexp{*}{\mathbf{R}}^{d_{ns}}$ the internal subset of  $\leftexp{*}{\mathbf{R}}^{d_{ns}}$ consisting of elements whose norm is equal to one. In [CN] the internal set $\mathbf{S}$ is referred to as the Wiener sphere. 

\bigskip

For $w \in \mathbf{S}$, define the internal polygonal path $\Sigma_{d_{ns}}(w): \leftexp{*}{ [0,1]} \rightarrow \leftexp{*}{\mathbf{R}}$ in a similar way as before: put $\tau_i = i/d_{ns}$ for $i=0,1,\cdots,d_{ns}$, then:
\begin{eqnarray*}
(\Sigma_{d_{ns}} (w))(0)  &=& 0 \\
(\Sigma_{d_{ns}} (w))(\tau_i)  &=& \sum_{k=1}^{i} w^{(k)} \mbox{ for } i=1,\cdots, d_{ns}
\end{eqnarray*}
and in general for $\tau \in \leftexp{*}{[0,1]}$, the value $(\Sigma_{d_{ns}} (w))(\tau)$ is linearly interpolated between $(i-1)/d_{ns}  \leq \tau \leq  i/d_{ns}$ for $i=1,\cdots,d_{ns} $.

\bigskip

Define the internal map:
\[
B: \mathbf{S} \times  \leftexp{*}{ [0,1]} \rightarrow \leftexp{*}{\mathbf{R}}
\]
by the rule: for $w \in \mathbf{S}$ and $\tau \in \leftexp{*}{[0,1]}$:
\[
B(w,\tau) = (\Sigma_{d_{ns}}(w))(\tau).
\]
Put $b(w,t) := \leftexp{\circ}{B(w,t)}$ for $w \in \mathbf{S}$ and $t \in [0,1]$. 

\bigskip

Now denote by $\mathbf{m}$ the uniform (i.e. rotationally invariant) internal probability measure on $\mathbf{S}$, and by $\mathbf{m}_L$ its Loeb extension; thus $\mathbf{S}$ equipped with the Loeb measure $\mathbf{m}_L$ is a probability space in the usual sense, and $b(\cdot,\cdot)$ is a stochastic process. 

\bigskip

We then have:
\begin{theorem} (Theorem 2.1 and Corollary 2.2 of [CN])
We have, for almost all $w \in \mathbf{S}$ (with respect to the Loeb measure $\mathbf{\mu}_L$), that the sample path $b(w,\cdot)$ defines an element in $C_0([0,1])$. Thus by restricting to a subset $\mathbf{S}^{\prime} \subset \mathbf{S}$ with $\mathbf{S} \backslash \mathbf{S}^{\prime}$ being of measure zero (with respect to $\mathbf{m}_L$), we have a map:
\begin{eqnarray*}
\Sigma_{\infty}: \mathbf{S}^{\prime} \rightarrow C_0([0,1]) 
\end{eqnarray*}
\begin{eqnarray*}
& & (\Sigma_{\infty}(w))(t) : = b(w,t) \\
&=& \leftexp{\circ}{\big( (\Sigma_{d_{ns}}(w))(t) \big) } ,   \,\ w \in \mathbf{S}^{\prime} , \,\  t \in [0,1].
\end{eqnarray*}
The map $\Sigma_{\infty}$ is measurable. Furthermore, the measure on $C_0([0,1])$ given by the push-forward of $\mathbf{m}_L$ by $\Sigma_{\infty}$:
\[
\mathbf{m}_L \circ (\Sigma_{\infty})^{-1}
\] 
is the Wiener measure on $C_0([0,1])$. In addition $b(\cdot,\cdot)$ is a standard Wiener process (whose distribution law is the Wiener measure). 
\end{theorem}

\bigskip

Note that a sample path $b(w,\cdot)$ for $w \in \mathbf{S}^{\prime}$, of the standard Wiener process $b(\cdot,\cdot)$, is $\Sigma_{\infty}(w)$. In addition, it also follows from Theorem 3.2 that, if $\{\mathbf{W}_n (\cdot )\}_{n \geq 0}$ is a sequence of independent identically distributed random variables on a probability space $\Omega$, with values in $ \mathbf{S}^{\prime}$ with uniform distribution with respect to the Loeb measure $\mathbf{m}_L$ of $\mathbf{S}$, then $\{b(\mathbf{W}_n(\cdot) ,\cdot) \}_{n \geq 0}$ is a sequence of independent standard Wiener processes, whose sample path sequence is $\{\Sigma_{\infty}(\mathbf{W}_n(\omega))\}_{n \geq 0}$, for $\omega \in \Omega$.

\bigskip

Now having recalled the results of [CN], we finally let $d$ be a large integer (simulation of a nonstandard infinite integer), and as in the end of the previous subsection, let $\{w_n\}$ be a sequence of pseudorandom vectors in $S^{d-1}$ with uniform distribution with respect to $S^{d-1}$ (simulation of a sample sequence of a sequence of independent identically distributed random variables, with values in $S^{d-1}$ with uniform distribution). For $n \geq 0$, define $\mathcal{B}_{n} = \Sigma_d(w_n) \in C_0([0,1])$. The previous discussion thus justifies the procedure of taking the sequence $\{ \mathcal{B}_{n} \}_{n \geq 0} \subset C_0([0,1])$ as discrete time simulation of uniformly distributed sample path sequence, of a sequence of independent standard Wiener processes (uniform distribution with respect to discrete time simulation of the Wiener measure).  

\begin{remark}
\end{remark}
\noindent The theorem of Cutland-Ng [CN] can be extended directly to the case of standard Wiener processes in $\mathbf{R}^{D}$ (for $D \in \mathbf{Z}_{\geq 1}$), by working with the Cartesian product of $D$ copies of the Wiener sphere $\mathbf{S}$. Our algorithm can thus be extended to this setting as well. Specifically, fix a large integer $d$ (again simulation of a nonstandard infinite integer). Let $\{u_n\}_{n \geq 0} \subset [0,1]^{D d }$ be sequence of uniform pseudorandom vectors as in section 2. Apply either the inverse transform method or the Box-Muller method to obtain sequence $\{v_n\}_{n \geq 0} \subset \mathbf{R}^{D d}$ of Gaussian pseudorandom vectors, with standard normal distribution on $\mathbf{R}^{Dd}$. With $v_n=(v_n^{(1)},\cdots, v_n^{(Dd)} ) \in \mathbf{R}^{D d}$, put for $j =1,\cdots,D$ and $n \geq 0$:
\[
v_{j,n} = ( v_n^{(1+ (j-1)d)} , \cdots, v_n^{(d+ (j-1)d)})  \in \mathbf{R}^d
\]
(thus in particular for each $j=1,\cdots,D$, the sequence $\{v_{j,n}\}_{n \geq 0} \subset \mathbf{R}^d$ is sequence of Gaussian pseudorandom vectors, with standard normal distribution on $\mathbf{R}^d$), and put $w_{j,n} = v_{j,n}/|| v_{j,n}|| \in S^{d-1}$. For $n \geq 0$, define the continuous piecewise linear map $\mathcal{B}^{D}_{n}: [ 0,1] \rightarrow \mathbf{R}^D$ by:
\[
\mathcal{B}^{D}_{n }(t) = \big((\Sigma_d(w_{1,n} ) )(t),\cdots,(\Sigma_d(w_{D,n}))(t) \big) \in \mathbf{R}^D, \,\ t \in [0,1].
\]
Then the sequence $\{\mathcal{B}^{D}_{ n }\}_{n \geq 0}$ gives discrete time simulation of uniformly distributed sample path sequence, of a sequence of independent standard Wiener processes in $\mathbf{R}^D$ (on the time interval $[0,1]$).

\section{Conclusion}

In this paper we present, using the arithmetic of elliptic curves over finite fields, an efficient algorithm for the generation of sequence uniform pseudorandom vectors in the unit hypercube. Criterion for the algorithm to generate sequences with maximum period is also given. 
\bigskip

We have shown how these could be transformed to construct, discrete time simulation of uniformly distributed sample path sequence, of a sequence of independent standard Wiener processes.  In a Monte Carlo style, these could be used for the numerical evaluation of expectation values against the Wiener measure, for example those occurring in Feynman-Kac type formulas. Examples of Monte Carlo integration, based on the algorithm of this paper, is the subject of the paper [MZ].

\bigskip
For the class of semilinear parabolic partial differential equations of the Kolmogorov type, their viscosity solutions have stochastic representation given by non-linear Feynman-Kac formulas ([BHJ]). In the full history recursive multi-level Picard approximation (MLP) method, these non-linear Feynman-Kac formulas could be evaluated numerically, using as input a denumerable set of independent standard Wiener processes (see for example [EHJK1], [EHJK2], [HJvW], [HK], [HJKNvW]). 

\bigskip
Our construction of discrete time simulation of uniformly distributed sample path sequence of a sequence of independent Wiener processes, using the algorithm as given in sections 2 and 3 of this paper, could thus be employed as inputs for the MLP method in the numerical approximation of solutions to these class of equations. Explicit numerical studies will be the subject of a future investigation. 

\section{Appendix}

In this appendix, we establish the inequalities (2.9), (2.10) and (2.11). There are two ingredients: Theorem 1 and Corollary 4 of [He], concerning the general discrepancy estimates with respect to the base $p$ Walsh function system, and the exponential sum estimates of [KS]. We first recall the former. 

\bigskip
As in section 2, we have $p$ is a prime. First recall the base $p$ Walsh functions $w_k: [0,1) \rightarrow \mathbf{C}^{\times}$ for $k \in \mathbf{Z}_{\geq 0}$. 

\bigskip
For $k \in \mathbf{Z}_{\geq 0}$, let
\[
k = \sum_{i=1}^{\infty} k(i)  p^{i-1}, \,\ k(i) \in \{0,1,\cdots,p-1\}
\]
be the unique expansion of $k$ in base $p$ (all but finitely many of the $k(i)$'s are equal to zero). Every number $\xi \in [0,1)$ has a unique base $p$ expansion:
\[
\xi =\sum_{i=1}^{\infty} \xi(i) p^{-i}, \,\ \xi(i) \in \{0,1,\cdots,p-1\}
\]
with the condition that $\xi(i) \neq p-1$ for infinitely many $i$. Then define:
\begin{eqnarray*}
w_k(\xi) = \exp\Big(\frac{2 \pi \sqrt{-1}}{p} \big(\sum_{i=1}^{\infty} k(i) \xi(i) \big)  \Big)
\end{eqnarray*}
(remark that [He] considers Walsh functions with respect to base that is not necessarily a prime, but this is good enough for our purpose).

\bigskip
More generally, for $h \in \mathbf{Z}_{\geq 1}$ and $ \mathbf{k} =(k^{(1)},\cdots,k^{(h)}) \in (\mathbf{Z}_{\geq 0})^h$, define for $\stackrel{\rightarrow}{\xi} =(\xi^{(1)},\cdots,\xi^{(h)}) \in [0,1)^h$: 
\[
w_{\mathbf{k}} (\stackrel{\rightarrow}{\xi}) = \prod_{\iota=1}^h w_{k^{(\iota)}}(\xi^{(\iota)}).
\]

Now let $ \mathcal{M}= \{\stackrel{\rightarrow}{\xi}_0,\cdots,\stackrel{\rightarrow}{\xi}_{L-1} \}$ be a set of $L$ distinct points in $[0,1)^h$, define:
\begin{eqnarray*}
S_{L}(w_{\mathbf{k}},\mathcal{M}) = \frac{1}{L}\big( w_{\mathbf{k}} (\stackrel{\rightarrow}{\xi}_0) + \cdots + w_{ \mathbf{k}} (\stackrel{\rightarrow}{\xi}_{L-1}) \big).
\end{eqnarray*}

\bigskip
Fix $a \in \mathbf{Z}_{\geq 1}$. We assume that $p^a \cdot \stackrel{\rightarrow}{\xi}_0,\cdots,  p^a \cdot \stackrel{\rightarrow}{\xi}_{L-1}$ belong to $\mathbf{Z}^h$. In addition define $\Delta_h \subset (\mathbf{Z}_{\geq 0})^h$ to be the following:
\begin{eqnarray*}
\Delta_h=\{ \mathbf{k} = (k^{(1)} ,\cdots, k^{(h)}) \in (\mathbf{Z}_{\geq 0})^h, \,\  0 \leq k^{(\iota)} < p^a  \mbox{ for } \iota=1,\cdots,h  \}.
\end{eqnarray*}
Put $\Delta^{*}_h = \Delta_h \backslash \{(0,\cdots,0)\}$.

\bigskip
We employ Theorem 1 of [He] to estimate the discrepancy of the set $\mathcal{M} \subset [0,1)^h$; in fact for our purpose, we only need to use Corollary 4 of {\it loc. cit.} in our present setting:

\begin{proposition} (Corollary 4 of [He]) With hypotheses as above, assume in addition that there is some for some $B \in \mathbf{R}_{\geq 0}$, such that for all $\mathbf{k} \in \Delta^{*}_h$, we have the bound:
\begin{eqnarray*}
|S_L(w_{\mathbf{k}} ,\mathcal{M} )| \leq B ,
\end{eqnarray*}
\noindent then the discrepancy $\mathcal{D}(\mathcal{M})$ of the set $\mathcal{M}=\{ \stackrel{\rightarrow}{\xi}_0, \cdots, \stackrel{\rightarrow}{\xi}_{L-1}  \} \subset [0,1)^h$ satisfies:
\begin{eqnarray}
\mathcal{D}(\mathcal{M}) \leq 1- \big(  1 - \frac{1}{p^a}   \big)^h + B \big(  2.43 \ln(p^a) +1     \big)^h.
\end{eqnarray}
\end{proposition} 

\bigskip
With $a$ as before, now as in section 2.3, let $ r \in \mathbf{Z}_{\geq 1 }$. Put $m=a \cdot r$ and $q=p^m$. Let $F$ be a finite field with cardinality $q$, and $K$ be the unique subfield of $F$ with cardinality $p^a=q^{1/r}$. Let $E$ be an elliptic curve over $F$, with $x,y$ being the affine Weierstrass coordinates, and $Q \in E(F)$ be a point of order $t$. Given an initial state $P_0 \in E(F)$, define $\{P_n\}_{n \geq 0}$ recursively by the rule: $P_{n+1} = P_n +Q$ for $n \geq 0$. Thus we have $P_n = [n](Q) + P_0$ and $P_{n+n^{\prime}} = [n](Q)  +P_{n^{\prime}}$ for $n,n^{\prime} \geq 0$, and $t$ is the period of the sequence $\{P_n\}_{n \geq 0}$. Fixing a basis $\mathfrak{a}=\{\kappa_1,\cdots,\kappa_a\}$ of the extension $K/\mathbf{F}_p$, with dual basis $\mathfrak{a}^{\prime} =\{\kappa_1^{\prime} ,\cdots,\kappa_a^{\prime}\}$ with respect to $Tr_{K/\mathbf{F}_p}$, and a basis $\mathfrak{b}=\{\lambda_1,\cdots,\lambda_r\}$ of the extension $F/K$ with dual basis $\mathfrak{b}^{\prime}= \{\lambda_1^{\prime}  ,\cdots,\lambda_r^{\prime}\}$ with respect to $Tr_{F/K}$, the output function $G: E(F) \rightarrow [0,1]^{2r}$ is defined as in section 2.3. From the definition of $G$, we have that $p^a \cdot G(P) \in \mathbf{Z}^{2r}$ for any $P \in E(F)$.

\bigskip

To establish (2.9), we take $h=2r$, and $\mathcal{M} \subset [0,1)^{2r}$ to be the set of points $G(P_n)$ for $n=0,1,\cdots,t-1$, with the point $G(P_n)$ being discarded if it is equal to $(1,\cdots,1)$ (i.e. if $P_n=\mathbf{O}$). The cardinality of $\mathcal{M}$ is equal to $t^{\prime}$ (which is equal to $t-1$ if $P_n=\mathbf{O}$ for some $n$, and is equal to $t$ otherwise). Applying Proposition 5.1, we see that to establish (2.9), it suffices to show:
\begin{eqnarray}
|S_{t^{\prime}}(w_{\mathbf{k}},\mathcal{M})| \leq \frac{4 q^{1/2}}{t^{\prime}},  \mbox{ for all } \mathbf{k} \in \Delta^{*}_{2r}.
\end{eqnarray}

\bigskip
To establish (5.2), we first compute the terms in the definition of $S_{t^{\prime}}(w_{\mathbf{k}},\mathcal{M})$. Recall that for $P \in E(F), P \neq \mathbf{O}$, the definition of the vector $G(P) \in [0,1)^{2r}$ is given by:
\begin{eqnarray*}
& & G(P) \\
&  = &\Big(\Phi(  \langle x(P) \rangle_1)  , \cdots, \Phi(  \langle x(P) \rangle_r) ,\Phi(  \langle y(P) \rangle_1),  \cdots ,\Phi(  \langle y(P) \rangle_r) \Big ).
\end{eqnarray*}
In addition for $1 \leq j \leq r$ and any $\eta \in F$, we have as in (2.8):
\begin{eqnarray*}
\Phi(\langle  \eta \rangle_j) = \sum_{i=1}^a \frac{Tr_{F/\mathbf{F}_p} (\eta \lambda^{\prime}_j \kappa_i^{\prime} ) }{p^i}.
\end{eqnarray*}

\bigskip
Thus for $\mathbf{k}= (k^{(1)},\cdots,k^{(2r)}  ) \in \Delta^{*}_{2r}$, and $P_n \neq \mathbf{O}$, we have: 
\begin{eqnarray*}
& & w_{\mathbf{k}}(G(P_n)) \\ &= &\exp \Big(       \frac{2 \pi \sqrt{-1}}{p}  \big(   \sum_{j=1}^r \sum_{i=1}^a   k^{(j)}(i) Tr_{F/\mathbf{F}_p}( x(P_n)   \lambda_j^{\prime} \kappa^{\prime}_i)   \big) \Big)  \\ & &   \times    \exp \Big(       \frac{2 \pi \sqrt{-1}}{p}  \big(      \sum_{j=1}^r \sum_{i=1}^a  k^{(r+j)}(i)  Tr_{F/\mathbf{F}_p}( y(P_n)   \lambda_j^{\prime} \kappa^{\prime}_i)   \big)    \Big) \\
&=& \exp \Big(  \frac{2 \pi \sqrt{-1}}{p}  Tr_{F/\mathbf{F}_p}  \big(   \eta_{\mathbf{k}} \cdot  x(P_n)  +   \overline{\eta}_{\mathbf{k}}\cdot   y(P_n)      \big) \Big)
\end{eqnarray*}
where
\begin{eqnarray}
\eta_{\mathbf{k}} =  \sum_{j=1}^r \sum_{i=1}^a   k^{(j)}(i)  \lambda_j^{\prime} \kappa^{\prime}_i , \,\ \overline{\eta}_{\mathbf{k}} =  \sum_{j=1}^r \sum_{i=1}^a   k^{(r+j)}(i)  \lambda_j^{\prime} \kappa^{\prime}_i
\end{eqnarray}
are elements of $F$. And as $\{ \lambda_j^{\prime} \kappa^{\prime}_i\}_{ 1\leq i \leq a , 1\leq j \leq r}$ is a basis of the extension $F/\mathbf{F}_p$, and $\mathbf{k}=(k^{(1)}, \cdots, k^{(2r)}) \neq (0,\cdots,0)$ we see that at least one of $\eta_{\mathbf{k}}, \overline{\eta}_{\mathbf{k}}$ is non-zero. 

\bigskip
Thus we see that to show (5.2), we need to establish for all $\mathbf{k} \in \Delta^{*}_{2r}$:
\begin{eqnarray}
\end{eqnarray}
\begin{eqnarray}
\Big|    \sum_{ \substack{   n=0,1,\cdots , t-1 \\ P_n \neq \mathbf{O}}}  \exp \Big(\frac{2 \pi \sqrt{-1}}{p}  Tr_{F/\mathbf{F}_p}  \big(  \eta_{\mathbf{k}}\cdot  x(P_n)   + \overline{\eta}_{\mathbf{k}} \cdot  y(P_n)   \big)    \Big)       \Big| \leq 4 q^{1/2}. \nonumber
\end{eqnarray}

\bigskip

Similarly to establish (2.10) and (2.11), consider $2 \leq s \leq t$ and take $h=2rs$. First consider the collection of points $\mathcal{N} \subset [0,1)^{2rs}$, given by the vectors:
\[
\big( G(P_n),G(P_{n+1}),\cdots,G(P_{n+s-1}) \big), \,\ 0 \leq n \leq t-1
\]
regarded as vectors in $[0,1]^{2rs}$, with the vector being discarded if one of the components $G(P_n),G(P_{n+1}),\cdots,G(P_{n+s-1})  $ is equal to $(1,\cdots,1)$. The cardinality of $\mathcal{N}$ is equal to $t^{\prime \prime}$. (which is equal to $t-s$ if $P_n=\mathbf{O}$ for some $n$, and is equal to $t$ otherwise) Applying again Proposition 5.1, we see that to establish (2.10), it suffices to show that for all $\mathbf{k} \in \Delta^{*}_{2rs}$:
\begin{eqnarray}
|S_{t^{\prime \prime}} (w_{\mathbf{k}},\mathcal{N}) | \leq \frac{6 q^{1/2} s}{t^{\prime \prime}}
\end{eqnarray}
under the condition that $p \geq 5$.

\bigskip
By a similar calculation as before, to establish (5.5) it amounts to showing that, under the condition $p \geq 5$, we have for all $\mathbf{k} =(k^{(1)}  ,\cdots,k^{(2rs)}) \in \Delta^{*}_{2rs}$:
\begin{eqnarray}
\end{eqnarray}
\begin{eqnarray}
    \Big|   & \sum & \exp \Big(\frac{2 \pi \sqrt{-1}}{p}  Tr_{F/\mathbf{F}_p}  \big(   \sum_{\iota=0}^{s-1} \eta^{(\iota)}_{\mathbf{k}}\cdot  x(P_{n+\iota})   + \sum_{\iota=0}^{s-1} \overline{\eta}^{(\iota)}_{\mathbf{k}} \cdot  y(P_{n+\iota})   \big)    \Big)       \Big|   \nonumber \\      & \leq &   6 q^{1/2} s    \nonumber
  \end{eqnarray}
here the summation is over $n=0,1,\cdots , t-1$ with $P_n,P_{n+1},\cdots,P_{n+s-1} \neq \mathbf{O}$, and where for $0 \leq \iota \leq s-1$:
\begin{eqnarray}
\,\ \,\ \,\ \,\  \,\ \eta_{\mathbf{k}}^{(\iota)} =  \sum_{j=1}^r \sum_{i=1}^a   k^{(2 r \iota+j)}(i)  \lambda_j^{\prime} \kappa^{\prime}_i , \,\ \overline{\eta}^{(\iota)}_{\mathbf{k}} =  \sum_{j=1}^r \sum_{i=1}^a   k^{( 2 r\iota+r+j)}(i)  \lambda_j^{\prime} \kappa^{\prime}_i
\end{eqnarray}
are elements of $F$. Again, since $\mathbf{k} = (k^{(1)},\cdots, k^{(2rs)}) \neq (0,\cdots,0)$, we have that at least one of the elements in the set $\{ \eta^{(\iota)}_{\mathbf{k}} ,\overline{\eta}^{(\iota)}_{\mathbf{k}} \}_{0 \leq \iota\leq s-1}$ is non-zero.

\bigskip
In an analogous manner, to establish (2.11), consider the collection of points $\widetilde{\mathcal{N}} \subset [0,1)^{2rs}$, given by the vectors:
\[
\big( G(P_{ns}),G(P_{ns+1}),\cdots,G(P_{ns+s-1}) \big), \,\ 0 \leq n \leq \frac{t}{\gcd(s,t)}-1
\]
regarded as vectors in $[0,1]^{2rs}$, with the vector being discarded if one of the components $G(P_{ns}),G(P_{ns+1}),\cdots,G(P_{ns+s-1})  $ is equal to $(1,\cdots,1)$. The cardinality of $\widetilde{\mathcal{N}}$ is equal to $t^{\prime \prime}/ \gcd(s,t)$. Applying Proposition 5.1, we see that to establish (2.11), it suffices to show that for all $\mathbf{k} \in \Delta^{*}_{2rs}$:
\begin{eqnarray}
|S_{t^{\prime \prime}/\gcd(s,t)} (w_{\mathbf{k}},\widetilde{\mathcal{N}}) | \leq \frac{6 q^{1/2} s^3}{t^{\prime \prime}}
\end{eqnarray}
under the conditions $p \geq 5$ and $\gcd(s,p)=1$.

\bigskip

Again by a similar calculation, to establish (5.8) it amounts to showing that, under the conditions $p \geq 5$ and $\gcd(s,p)=1$, we have for all $\mathbf{k} =(k^{(1)}  ,\cdots,k^{(2rs)}) \in \Delta^{*}_{2rs}$:
\begin{eqnarray}
\end{eqnarray}
\begin{eqnarray}
    \Big|    &\sum & \exp \Big(\frac{2 \pi \sqrt{-1}}{p} Tr_{F/\mathbf{F}_p}    \big(   \sum_{\iota=0}^{s-1} \eta^{(\iota)}_{\mathbf{k}}\cdot  x(P_{ns+\iota})   + \sum_{\iota=0}^{s-1} \overline{\eta}^{(\iota)}_{\mathbf{k}} \cdot  y(P_{ns+\iota})   \big)    \Big)       \Big|   \nonumber \\        &\leq &   6 q^{1/2} \frac{s^3}{\gcd(s,t)}    \nonumber
  \end{eqnarray}
here the summation is over $n =0,1,\cdots,\frac{t}{\gcd(s,t)}-1$ with $ P_{ns},P_{ns+1},\cdots,P_{ns+s-1} \neq \mathbf{O}$ (and where $\eta_{\mathbf{k}}^{(\iota)}, \overline{\eta}_{\mathbf{k}}^{(\iota)}$ are as in (5.7)).

\bigskip

Thus to complete the proofs of (2.9), (2.10) and (2.11), it remains to establish (5.4), (5.6) and (5.9). But these follow from the exponential sums estimates of [KS] (which extends the results of section VI of [Bo]); specifically we use Corollary 1 of [KS]. 

\bigskip
Firstly for an element $f$ in the function field of $E$ over $F$, we say that $f$ satisfies condition (A), if for any element $\mathcal{G}$ in the function field of $E$ over $\overline{F}$ (algebraic closure of $F$), we have $f \neq \mathcal{G}^p-\mathcal{G}$. 

\begin{proposition}(Corollary 1 of [KS]) Let $f$ be a nonconstant element in the function field of $E$ over $F$, satisfying condition (A), and $\mathcal{H} \subset E(F)$ be a subgroup. Then we have:
\begin{eqnarray}
\,\  \,\ \Big|\sum_{\substack{P \in \mathcal{H} \\f(P) \neq \infty}  }  \exp \Big(\frac{2 \pi \sqrt{-1}}{p}   Tr_{F/\mathbf{F}_p} \big(   f(P)        \big)     \Big) \Big| \leq 2 \deg(f) q^{1/2}.
\end{eqnarray}
In addition if the polar divisor of $f$ has support at a single prime divisor, then we have the stronger bound:
\begin{eqnarray}
 \,\ \,\ \,\ \,\ \,\ \,\ \Big|\sum_{\substack{P \in \mathcal{H} \\f(P) \neq \infty}  }  \exp \Big(\frac{2 \pi \sqrt{-1}}{p}   Tr_{F/\mathbf{F}_p} \big(   f(P)        \big)     \Big) \Big| \leq (1+ \deg(f)) q^{1/2}. 
\end{eqnarray}
\end{proposition}

To apply Proposition 5.2 in our present context, we take $\mathcal{H}$ to be the cyclic subgroup of $E(F)$ generated by the point $Q$; thus $\mathcal{H} = \{ [n](Q), n=0,1,\cdots,t-1 \}$. For any point $P \in E(F)$, we denote by $\tau_{P} : E \rightarrow E$ the translation by $P$ map on $E$ (thus $\tau_P(R) = R+P$ for any point $R$ on $E$); we have $\tau_P$ is an automorphism of $E$ (as a genus one curve) over $F$.

\bigskip

To establish (5.4), we put for $\mathbf{k} \in \Delta_{2r}^{*}$:
\[
f_{\mathbf{k}} = \eta_{\mathbf{k}} \cdot x + \overline{\eta}_{\mathbf{k}} \cdot y
\]
The polar divisor of $x$ (respectively $y$) is supported at $\mathbf{O}$ with multiplicity $2$ (respectively $3$). Since at least one of the two elements $\eta_{\mathbf{k}},\overline{\eta}_{\mathbf{k}}$ of $F$ is non-zero, it follows that $f_{\mathbf{k}}$ is nonconstant, whose polar divisor is again supported at $\mathbf{O}$, with multiplicity $\leq 3$. Thus in particular $\deg (f_{\mathbf{k}}) \leq 3$. Put:
\[
g_{\mathbf{k}}= f_{\mathbf{k}} \circ \tau_{P_0}
\] 
(so in particular $g_{\mathbf{k}}([n](Q))  = f_{\mathbf{k}}(P_n)$). Then $g_{\mathbf{k}}$ is nonconstant, the polar divisor of $g_{\mathbf{k}}$ is supported at $-P_0$, with multiplicity at most $3$, and $\deg(g_{\mathbf{k}}) = \deg(f_{\mathbf{k}}) \leq 3$. Note also that the condition $g_{\mathbf{k}}([n](Q)) \neq \infty$ is exactly the condition $P_n \neq \mathbf{O}$. Thus the exponential sum occurring on the left hand side of (5.4) is exactly equal to:
\[
\sum_{\substack{P \in \mathcal{H} \\ g_{\mathbf{k}}(P) \neq \infty}  }  \exp \Big(\frac{2 \pi \sqrt{-1}}{p}   Tr_{F/\mathbf{F}_p} \big(   g _{\mathbf{k}}(P)        \big)     \Big)
\]
and hence the estimate (5.4) follow from the bound (5.11) of Proposition 5.2, applied to $f= g_{\mathbf{k}}$, if we can show that $g_{\mathbf{k}}$ satisfies condition (A). Indeed if we were to have $g_{\mathbf{k}}= \mathcal{G}^p-\mathcal{G}$ for some $\mathcal{G}$ in the function field of $E$ over $\overline{F}$, then (working over $\overline{F}$) we see that $\mathcal{G}$ must be nonconstant, whose polar divisor coincides with that of $g_{\mathbf{k}}$, and hence is supported at the point $-P_{0}$. By the Riemann-Roch Theorem for elliptic curves, the multiplicity of the pole of $\mathcal{G}$ at $-P_0$ is then at least $2$, and thus the multiplicity of the pole of $\mathcal{G}^p - \mathcal{G}$ at $-P_0$ is at least $2p \geq 4$. But we have already seen that the multiplicity of $g_{\mathbf{k}}$ at $-P_{0}$ is at most $3$; thus $g_{\mathbf{k}}$ satisfies condition (A).
 
\bigskip
To establish (5.6), we put for $\mathbf{k} \in \Delta_{2rs}^{*}$ and $0 \leq \iota \leq s-1$:
\begin{eqnarray*}
f^{(\iota)}_{\mathbf{k}} &=& \eta^{(\iota)}_{\mathbf{k}} \cdot x + \overline{\eta}^{(\iota)}_{\mathbf{k}} \cdot y \\
g^{(\iota)}_{\mathbf{k}}&=&  f^{(\iota)}_{\mathbf{k}} \circ \tau_{P_{\iota}}
\end{eqnarray*}
then similarly, either $f^{(\iota)}_{\mathbf{k}}$ is identically equal to zero (which happens only when both $\eta^{(\iota)}_{\mathbf{k}}$ and $\overline{\eta}^{(\iota)}_{\mathbf{k}}$ are zero), or else is nonconstant, whose polar divisor is supported at $\mathbf{O}$, with $\deg(f_{\mathbf{k}}^{(\iota)}) \leq 3$. In addition since at least one of the elements in the set $\{\eta^{(\iota)}_{\mathbf{k}},  \overline{\eta}^{(\iota)}_{\mathbf{k}}\}_{0 \leq \iota \leq s-1}$ is non-zero, we have $f^{(\iota)}_{\mathbf{k}}$ is not identically zero (and hence nonconstant) for some $0 \leq \iota \leq s-1$.

\bigskip

Thus it also follows that either $g^{(\iota)}_{\mathbf{k}}$ is identically equal to zero (which happens only when both $\eta^{(\iota)}_{\mathbf{k}}$ and $\overline{\eta}^{(\iota)}_{\mathbf{k}}$ are zero), or else is nonconstant, whose polar divisor is supported at $-P_{\iota}$, with $\deg(g_{\mathbf{k}}^{(\iota)}) \leq 3$. In addition we have $g^{(\iota)}_{\mathbf{k}}$ is not identically zero (and hence nonconstant) for some $0 \leq \iota \leq s-1$. Put:
\[
\mathfrak{g}_{\mathbf{k}} = \sum_{\iota=0}^{s-1} g^{(\iota)}_{\mathbf{k}}
\]
(so in particular $\mathfrak{g}_{\mathbf{k}} ([n](Q)) =   \sum^{s-1}_{\iota=0} g^{(\iota)}_{\mathbf{k}}([n](Q))  =\sum^{s-1}_{\iota=0} f_{\mathbf{k}}^{(\iota)}(P_{n+\iota})$). Noting that the points $-P_0,-P_1,\cdots,-P_{s-1}$, are all distinct, it then follows that $\mathfrak{g}_{\mathbf{k}}$ is nonconstant, whose polar divisor is supported among the points $\{-P_{\iota}\}_{0 \leq \iota \leq s-1}$, and the multiplicity of pole at each point $-P_{\iota}$ is $\leq 3$. In particular we have $\deg(\mathfrak{g}_{\mathbf{k}}) \leq 3s$. Note also that the condition $\mathfrak{g}_{\mathbf{k}}([n](Q)) \neq \infty$ is exactly the condition $P_{n},P_{n+1},\cdots,P_{n+s-1} \neq \mathbf{O}$. Thus the exponential sum occurring on the left hand side of (5.6) is exactly equal to:
\[
\sum_{\substack{P \in \mathcal{H} \\ \mathfrak{g}_{\mathbf{k}}(P) \neq \infty}  }  \exp \Big(\frac{2 \pi \sqrt{-1}}{p}   Tr_{F/\mathbf{F}_p} \big(   \mathfrak{g} _{\mathbf{k}}(P)        \big)     \Big)
\]
and hence the estimate (5.6) follow from the bound (5.10) of Proposition 5.2, applied to $f= \mathfrak{g}_{\mathbf{k}}$, if we can show that $\mathfrak{g}_{\mathbf{k}}$ satisfies condition (A) under the condition $p \geq 5$. Indeed if we were to have $\mathfrak{g}_k = \mathcal{G}^p - \mathcal{G}$ for some element $\mathcal{G}$ in the function field of $E$ over $\overline{F}$, then (working over $\overline{F}$) we see similarly that $\mathcal{G}$ is nonconstant, whose polar divisor coincides with that of $\mathfrak{g}_{\mathbf{k}}$, and then the multiplicity of each pole of $\mathcal{G}^p - \mathcal{G}$ at least $p$. On the other hand the multiplicity of each pole of $\mathfrak{g}_{\mathbf{k}}$ is at most $3$. It follows that $\mathfrak{g}_{\mathbf{k}}$ satisfies condition (A) when $p \geq 5$.

\bigskip

Finally to establish (5.9), we put for $\mathbf{k} \in \Delta_{2rs}^{*}$:
\[
\mathfrak{h}_{\mathbf{k}} = \mathfrak{g}_{\mathbf{k}} \circ [s]
\]
(so in particular $\mathfrak{h}_{\mathbf{k}}([n](Q)) =\mathfrak{g}_{\mathbf{k}}([ns](Q))$). By Chapter III, Theorem 6.2(d) of [Si], the multiplication by $s$ map $[s]$ is an isogeny from $E$ to $E$ over $F$, and $\deg([s]) =s^2$. It follows that $\mathfrak{h}_{\mathbf{k}}$ is also nonconstant, and
\[
\deg(\mathfrak{h}_{\mathbf{k}}) = \deg(\mathfrak{g}_{\mathbf{k}}) \cdot \deg([s]) = \deg(\mathfrak{g}_{\mathbf{k}}) \cdot s^2 \leq 3s^3.
\]  
In addition we can again show that $\mathfrak{h}_{\mathbf{k}}$ satisfies condition (A) when $p \geq 5$ and $\gcd(s,p)=1$; see below.

\bigskip

Now apply the bound (5.10) of Proposition 5.2 to $\mathfrak{h}_{\mathbf{k}}$, we then obtain:
\begin{eqnarray}
\,\  \,\ \,\  \,\  \,\ \,\   \,\ \Big|\sum_{\substack{ n=0,1,\cdots,t-1 \\\mathfrak{g}_{\mathbf{k}}([ns](Q)) \neq \infty}  }  \exp \Big(\frac{2 \pi \sqrt{-1}}{p}   Tr_{F/\mathbf{F}_p} \big(   \mathfrak{g}_{\mathbf{k}}( [ns](Q))        \big)     \Big) \Big| \leq  6 q^{1/2}s^3. 
\end{eqnarray}
Noting that as a function of $n$, the term $\mathfrak{g}_{\mathbf{k}}([ns](Q)) $ is periodic modulo $t/\gcd(s,t)$, we see that the exponential sum occurring on the left hand side of (5.12) is equal to:
\begin{eqnarray}
 \gcd(s,t) \cdot \sum_{\substack{ n=0,1,\cdots, \frac{t}{\gcd(s,t)}-1 \\\mathfrak{g}_{\mathbf{k}}([ns](Q)) \neq \infty}  }  \exp \Big(\frac{2 \pi \sqrt{-1}}{p}   Tr_{F/\mathbf{F}_p} \big(   \mathfrak{g}_{\mathbf{k}}( [ns](Q))        \big)     \Big) . \nonumber
\end{eqnarray}
In addition, the condition $\mathfrak{g}_{\mathbf{k}}([ns](Q)) \neq \infty$ is exactly the condition $P_{ns},P_{ns+1},\cdots,P_{ns+s-1} \neq \mathbf{O}$. So (5.12) can be rewritten as:
\begin{eqnarray}
    \Big|    &\sum_{ \substack{ n =0,1,\cdots,\frac{t}{\gcd(s,t)}-1 \\ P_{ns},P_{ns+1},\cdots,P_{ns+s-1} \neq \mathbf{O}}} & \exp \Big(\frac{2 \pi \sqrt{-1}}{p} Tr_{F/\mathbf{F}_p}    \big(  \mathfrak{g}_{\mathbf{k}}([ns](Q))  \big)    \Big)       \Big|   \nonumber \\        &\leq &   6 q^{1/2} \frac{s^3}{\gcd(s,t)}    \nonumber
  \end{eqnarray}
But this is exactly (5.9).

\bigskip
It remains to show that $\mathfrak{h}_{\mathbf{k}}$ satisfies condition (A) when $p \geq 5$ and $\gcd(s,p)=1$. Working over $\overline{F}$, first note that as $\gcd(s,p)=1$, we have that the isogeny $[s]$ is unramified (Chapter III, Corollary 5.4 of [Si]). Thus as the multiplicity of each pole of $\mathfrak{g}_{\mathbf{k}}$ is at most $3$, it follows that the multiplicity of each pole of $ \mathfrak{h}_{\mathbf{k}} =\mathfrak{g}_{\mathbf{k}} \circ [s ]$ is also at most $3$. Now if we were to have $\mathfrak{h}_{\mathbf{k}} =\mathcal{G}^p -\mathcal{G}$ for some element $\mathcal{G}$ in the function field of $E$ over $\overline{F}$, then again $\mathcal{G}$ is nonconstant, whose polar divisor coincides with that of $\mathfrak{h}_{\mathbf{k}}$, and then the multiplicity of each pole of $\mathcal{G}^p - \mathcal{G}$ is at least $p$. It again follows that $\mathfrak{h}_{\mathbf{k}}$ satisfies condition (A) when $p \geq 5$ and $\gcd(s,p)=1$.

\bigskip
This completes the proofs of (2.9), (2.10) and (2.11).

\end{document}